\documentclass{article}

\usepackage{amssymb}
\usepackage{amsmath}
\usepackage{amsfonts}

\newtheorem{theorem}{Theorem}

\newtheorem{corollary}[theorem]{Corollary}

\newtheorem{example}[theorem]{Example}

\newtheorem{examples}[theorem]{Examples}
\newtheorem{lemma}[theorem]{Lemma}

\newtheorem{proposition}[theorem]{Proposition}
\newtheorem{remark}[theorem]{Remark}

\newenvironment{proof}[1][Proof]{\noindent\textbf{#1.} }{\ \rule{0.5em}{0.5em}}

\begin{document}

\title{On the Distribution of Explosion Time of Stochastic Differential
Equations}
\author{Jorge A. Le\'{o}n, Liliana Peralta Hern\'{a}ndez \\
Departamento de Control Autom\'{a}tico\\
Cinvestav-IPN \\
Apartado postal 14-740\\
07000 M\'exico D.F., Mexico\\
and \and Jos\'{e} Villa-Morales \\
Departamento de Matem\'{a}ticas y F\'{\i}sica\\
Universidad Aut\'{o}noma de Aguascalientes\\
Avenida Universidad 940, Ciudad Universitaria \\
20131 Aguascalientes, \ Ags., Mexico}
\date{}
\maketitle

\begin{abstract}
In this paper we use the It\^o's formula and comparison theorems to study
the blow-up in finite time of stochastic differential equations driven by a
Brownian motion. In particular, we obtain an extension of Osgood criterion,
which can be applied to some nonautonomous stochastic differential equations
with additive Wiener integral noise. In most cases we are able to provide
with a method to figure out the distribution of the explosion time of the
involved equation.
\end{abstract}




\section{Introduction}

Consider the stochastic differential equation 
\begin{eqnarray}
dX_{t} &=&b(X_{t})dt+\sigma (X_{t})dW_{t},\quad t>0,  \label{eq:facil} \\
X_{0} &=&x_{0}.  \notag
\end{eqnarray}%
Here $b,\sigma :{\mathbb{R}}\rightarrow {\mathbb{R}}$ are two locally
Lipschitz functions, $x_{0}\in {\mathbb{R}}$ and $\{W_{t}:t\geq 0\}$ is a
Brownian motion defined on a complete probability space $(\Omega ,\mathcal{F}%
,P)$.

It is well-known that the solution $X$ of equation (\ref{eq:facil}) may
explode in finite time. That is, $|X_{t}|$ goes to infinite as $t$
approaches to a stopping time that could be finite with positive
probability, which is called the explosion time of equation (\ref{eq:facil})
(see McKean \cite{MK}). The Feller test is an important tool of the
stochastic calculus to know if there is blow-up in finite time for (\ref%
{eq:facil}) (see, for example, Karatzas and Shreve \cite{KaS}). The reader
can consult de Pablo et al. \cite{dP} (and references therein) for
applications of blow-up.

In the case that $b$ is non-decreasing and positive, and $\sigma \equiv 1$,
Feller test is equivalent to Osgood criterion \cite{Os}, as it is proven in
Le\'{o}n and Villa \cite{L-V}. It means, the solution of (\ref{eq:facil})
explodes in finite time if and only if $\int_{x_{0}}^{\infty
}(1/b(s))ds<\infty $. Also, when $\sigma \equiv 0$ and $b>0$, Osgood \cite%
{Os} has stated that explosion time is finite if and only if $%
\int_{x_{0}}^{\infty }(1/b(s))ds<\infty $. In this case, the explosion time
is equals to this integral.

Unfortunately, the distribution of the explosion time of equation (\ref%
{eq:facil}) is not easy to calculate. One way to do it is using linear
second-order ordinary differential equations. Indeed, Feller \cite{Wfe} has
pointed out the Laplace transformation of this distribution is a bounded
solution to some related ordinary differential equations (see Section \ref%
{sec:5.2} below for a generalization of this result). Also some numerical
schemes have been analyzed in order to approximate the time of explosion
(consult D\'avila et al. \cite{Da}). In this paper, in Section \ref%
{subsec:5.1}, we also obtain the partial differential equation that has the
distribution of the explosion time as a bounded solution.

Now consider the nonautonomous stochastic differential equation 
\begin{eqnarray}
dX_{t} &=&b(t,X_{t})dt+\sigma (t,X_{t})dW_{t},\quad t>0,  \label{eq:dificil}
\\
X_{0} &=&x_{0}.  \notag
\end{eqnarray}%
For this equation, Feller test and Osgood criterion are not useful anymore,
but, in the case that $\sigma $ is independent of $x$, we are still able to
associate the Laplace transformation of the distribution of the explosion
time of (\ref{eq:dificil}) with a partial differential equation as Theorem %
\ref{thm:22} below establishes.

The main purpose of this paper is to deal with some extensions of Osgood
criterion for some equations of the form (\ref{eq:dificil}). For instance,
Lemma \ref{lem:n7} provides a better understanding of Theorem 2.1 in \cite%
{Co}, or if, in (\ref{eq:dificil}), $\sigma $ is independent of $x$, we
obtain an extension of Osgood criterion by means of the law of iterated
logarithm and comparison theorems. It is worth mentioning that versions of
these important tools have been used to analyze global solutions of integral
equations as it is done by Constantin \cite{Co}, or to obtain an extension
of Osgood criterion to integral equations with additive noise and with $%
0<b(t,x)=b(x)$ non-decreasing (see Le\'{o}n and Villa \cite{L-V}).

The paper is organized as follows. Our comparison theorem for integral
equations is introduced in Section \ref{sec:3}. Some extensions of Osgood
criterion are given is Sections \ref{sec:2}, \ref{sec:3} and \ref{sec:4}.
Finally, the relation between partial differential equations and finite
blow-up is considered in Section \ref{sec:5}.

\section{Osgood criterion for some stochastic differential equation with
diffusion coefficient}

\label{sec:2}

Let $\sigma:{\mathbb{R}}\rightarrow{\mathbb{R}}$ and $h:{\mathbb{R}}%
\rightarrow{\mathbb{R}}$ be a differentiable function and a continuous
function, respectively. We consider the stochastic differential equation 
\begin{equation}
X_{t}^{\xi }=\xi +\frac{1}{2}\int_{0}^{t}\sigma (X_{s}^{\xi })\sigma
^{\prime }(X_{s}^{\xi })h^2(s)ds+\int_{0}^{t}\sigma (X_{s}^{\xi })
h(s)dW_{s}, \quad t\ge 0,  \label{EDEa}
\end{equation}%
where $\xi\in{\mathbb{R}}$. Here and in what follows, $W=\{W_t: t\ge0\}$ is
a Brownian motion.

Now we assume that there are $-\infty \leq x_{1}<x_{2}\leq \infty $ such
that $\sigma \neq 0$ on $(x_{1},x_{2})$. Let $\xi \in (x_{1},x_{2})$ be
fixed and define $\Psi _{\xi }:(x_{1},x_{2})\rightarrow \mathbb{R}$ as%
\begin{equation*}
\Psi _{\xi }(x)=\int_{\xi }^{x}\frac{dz}{\sigma (z)}.
\end{equation*}%
Set $l_{\xi }=\Psi _{\xi }(x_{1})\wedge \Psi _{\xi }(x_{2})$, $r_{\xi }=\Psi
_{\xi }(x_{1})\vee \Psi _{\xi }(x_{2})$ and $Y_{t}=\int_{0}^{t}h(s)dW_{s}$, $%
t\geq 0.$

The following result is our first extension of Osgood criterion.

\begin{theorem}
\label{the:v1} Let $\tau _{\xi }=\inf \{t\geq 0:Y_{t}\notin (l_{\xi },r_{\xi
})\}$. Then, the process $X_{t}^{\xi }=\{\Psi _{\xi }^{-1}(Y_{t}):0\leq
t<\tau _{\xi }\}$ is a solution of equation (\ref{EDEa}).
\end{theorem}

\begin{remark}
In this case, $\tau _{\xi}$ is called the explosion time of the solution to
equation (\ref{EDEa}).
\end{remark}

\begin{proof}
Applying It\^{o}'s formula with $f(x)=\Psi _{\xi }^{-1}(x)$, $x\in (l_{\xi
},r_{\xi })$ we have%
\begin{equation*}
f(Y_{t\wedge \tau _{\xi }^{k}})-f(0)=\frac{1}{2}\int_{0}^{t\wedge \tau _{\xi
}^{k}}f^{\prime \prime }(Y_{s})h^{2}(s)ds+\int_{0}^{t\wedge \tau _{\xi
}^{k}}f^{\prime }(Y_{s})h(s)dW_{s},  \label{famitoaos}
\end{equation*}%
where%
\begin{equation*}
\tau _{\xi }^{k}=\inf \{t>0:Y_{t}\notin (l_{\xi }+k^{-1},r_{\xi }-k^{-1})\}.
\end{equation*}%
Letting $k\rightarrow \infty $ in (\ref{famitoaos}) we get the result holds. 
$\hfill $
\end{proof}

An immediate consequence of Theorem \ref{the:v1} is the following:

\begin{corollary}
\label{cor:dis} Let $\int_{0}^{\infty }h^{2}(s)ds=\infty $. Then the
solution of equation (\ref{EDEa}) explodes in finite time if and only if
either $l_{\xi }>-\infty $, or $r_{\xi }<\infty $. Moreover, if $l_{\xi }$
and $r_{\xi }$ are two real numbers, then 
\begin{equation*}
P(\tau _{\xi }\in dt)=\sum\limits_{k=-\infty }^{\infty }(-1)^{k}\frac{r_{\xi
}+k(r_{\xi }-l_{\xi })}{\sqrt{2\pi }(H(t))^{3/2}}\exp \left( -\frac{(r_{\xi
}+k(r_{\xi }-l_{\xi }))^{2}}{2H(t)}\right) dt,
\end{equation*}%
with $H(t)=\int_{0}^{t}(h(s))^{2}ds$.
\end{corollary}

\begin{proof}
It is well-known that there is a Brownian motion $B=\{B_t:t\ge 0\}$ such
that $Y_{t}=B_{H(t)}$, $t\ge0$, (see, for instance, Durrett \cite%
{Du}). Let ${\tilde\tau}_{\xi}=\inf \{t>0: B_t\notin(l_{\xi},r_{\xi})\}$.
Then, it is easy to show that $P(\tau_{\xi}\le t)= P({\tilde\tau}_{\xi}\le
H(t))$. Consequently, the proof follows from Borodin and Salminen \cite{B-S}
(page 212).$\hfill$
\end{proof}

\begin{remark}
Suppose that, for example, $\sigma >0$, $\Psi _{\xi }(x_{1})=-\infty $ and $%
\Psi _{\xi }(x_{2})<\infty $. Then, as an immediate consequence of the proof
of Corollary \ref{cor:dis}, we get that $\tau _{\xi }=\inf
\{t:\int_{0}^{t}h(s)dW_{s}=\Psi _{\xi }(x_{2})\}$ and 
\begin{equation}
P(\tau _{\xi }\leq t)=\Phi \left( \frac{\Psi _{\xi }(x_{2})}{\sqrt{H(t)}}%
\right) ,  \label{disttiemplleBro}
\end{equation}%
where 
\begin{equation*}
\Phi (x)=\frac{2}{\sqrt{2\pi }}\int_{x}^{\infty }e^{-z^{2}/2}dz.
\end{equation*}%
Observe that we get a similar result when $\sigma $ is negative, or the
involved interval has the form $(l_{\xi },\infty ).$
\end{remark}

Now we illustrate this remark with two examples.

\begin{example}
\label{ExamOSG1}Let $\sigma (x)=|x|^{\alpha }$, $x\in \mathbb{R}$, $\alpha
>1 $ and $\xi \in \mathbb{R}$. Then

\begin{equation*}
\Psi _{\xi }(x)=\left\{ 
\begin{array}{cc}
\frac{1}{1-\alpha }(|x|^{1-\alpha }-|\xi|^{1-\alpha }), & \xi>0,\ x\ge0, \\ 
\frac{1}{1-\alpha }(|\xi|^{1-\alpha }-|x|^{1-\alpha }), & \xi<0,\ x\le 0.%
\end{array}%
\right.
\end{equation*}%
Hence, 
\begin{equation*}
\Psi_{\xi}(-\infty)=\frac{|\xi|^{1-\alpha}}{1-\alpha}\quad \hbox{\rm and}%
\quad \Psi_{\xi}(0)=\infty,\quad\hbox{\rm for}\ \xi<0,
\end{equation*}
and 
\begin{equation*}
\Psi_{\xi}(\infty)=\frac{|\xi|^{1-\alpha}}{\alpha-1}\quad \hbox{\rm and}%
\quad \Psi_{\xi}(0)=-\infty,\quad\hbox{\rm for}\ \xi>0.
\end{equation*}

Therefore, there is explosion in finite time and 
\begin{equation*}
P(\tau _{\xi }\leq t)=\Phi \left( \frac{|\xi |^{1-\alpha }}{(\alpha -1)\sqrt{%
H(t)}}\right) .
\end{equation*}
\end{example}

\begin{example}
\label{ExamOSG2}Let $\sigma (x)=e^{\alpha x}$, $x\in \mathbb{R}$, $\alpha
\neq 0$ and $\xi \in \mathbb{R}$. Then%
\begin{equation*}
\Psi _{\xi }(x)=\frac{1}{\alpha }(e^{-\alpha \xi }-e^{-\alpha x}),
\end{equation*}%
\begin{equation*}
\Psi _{\xi }(-\infty )=\left\{ 
\begin{array}{ll}
-\infty , & \alpha >0, \\ 
\frac{1}{\alpha }e^{-\alpha \xi }, & \alpha <0,%
\end{array}%
\right. \text{ \ \ and \ \ }\Psi _{\xi }(\infty )=\left\{ 
\begin{array}{ll}
\frac{1}{\alpha }e^{-\alpha \xi }, & \alpha >0, \\ 
\infty , & \alpha <0.%
\end{array}%
\right.
\end{equation*}%
Thus we deduce that there is explosion on the left for $\alpha <0$, there is
explosion on the right for $\alpha >0$ and 
\begin{equation*}
P(\tau _{\xi }\leq t)=\Phi \left( \frac{e^{-\alpha \xi }}{|\alpha |\sqrt{H(t)%
}}\right) .
\end{equation*}
\end{example}

\section{An extension of Osgood criterion for integral equations}

\label{sec:3}

In this section we generalize recent results obtained in \cite{M-J-J} and 
\cite{L-V}. Now we study the following nonautonomous integral equation 
\begin{equation}
X_{t}^{\xi }=\xi +\int_{0}^{t}a(s)b(X_{s}^{\xi })ds+g(t), \quad t\ge0.
\label{eqdnhomo}
\end{equation}
The explosion time $T_{\xi}^X$ of this equation is defined as $%
T_{\xi}^X=\inf \{t\ge 0: X_t^{\xi}\notin {\mathbb{R}}\}$. In the remaining
of this paper we will need the following conditions:

\begin{description}
\item[H1:] $a:(0,\infty )\rightarrow (0,\infty )$ is a continuous function
such that 
\begin{equation*}
\lim_{t\rightarrow \infty }\int_{t}^{t+\eta }a(s)ds>0, \quad%
\hbox{\rm for
some}\ \eta>0.
\end{equation*}

\item[H2:] $b:\mathbb{R}\rightarrow \lbrack 0,\infty )$ is a continuous
function such that there exist $-\infty \leq l<\infty $ and $-\infty
<r<\infty $ satisfying that $b>0$ and locally Lipschitz on $(l,\infty )$,
and $b:[r,\infty )\rightarrow (0,\infty )$ is non-decreasing.

\item[H3:] $g:[0,\infty )\rightarrow \mathbb{R}$ is a continuous function
such that 
\begin{equation*}
\limsup_{t\rightarrow \infty }\left( \inf_{0\leq h\leq {\tilde{\eta}}%
}g(t+h)\right) =\infty ,\quad \hbox{\rm for some}\ {\tilde{\eta}}>0.
\end{equation*}
\end{description}

Henceforth we utilize the convention 
\begin{equation*}
A_t(x)=\int_{t}^{x}a(z)dz,\quad t\ge0\ \hbox{\rm and }x\in(t,\infty ),
\end{equation*}
and 
\begin{equation*}
B_{\xi }(x)=\int_{\xi }^{x}\frac{dz}{b(z)},\quad x\in (l,\infty).
\end{equation*}

We begin with the following generalization of Osgood criterion.

\begin{lemma}
\label{lem:n7} Let \textbf{H1} and \textbf{H2} be satisfied and $x_{0}>l$.
Consider the ordinary differential equation 
\begin{eqnarray}
\frac{dy(t)}{dt} &=&a(t)b(y(t))dt,\ \ t>t_0,  \label{inrtosg} \\
y(t_0) &=&x_{0}.  \notag
\end{eqnarray}

\begin{itemize}
\item[a)] Assume that $B_{x_{0}}(\infty )\geq A_{t_0}(\infty )$, then 
\begin{equation*}
y(t)=B_{x_{0}}^{-1}(A_{t_0}(t)),\ \ t\ge t_0.
\end{equation*}

\item[b)] If $B_{x_{0}}(\infty )<A_{t_0}(\infty )$, then there is blow up in
finite time and the time of explosion $T_{x_{0}}^{y}$ is equal to $%
A_{t_0}^{-1}(B_{x_{0}}(\infty ))$.
\end{itemize}
\end{lemma}

\begin{remark}
Observe that equation (\ref{inrtosg}) (resp. equation (\ref{eqdnhomo})) has
a unique solution for $x_0>l$ (resp. for $\xi>l$) that may explode in finite
time because of Hypotheses \textbf{H1} and \textbf{H2} (resp. \textbf{H1}-%
\textbf{H3}). This fact will be used in the proof of Theorem \ref%
{TheOsnoiGral} below without mentioning
\end{remark}

\begin{proof}
From (\ref{inrtosg}) we see that 
\begin{equation*}
\int_{t_0}^{t}\frac{y^{\prime }(s)}{b(y(s))}ds=\int_{t_0}^{t}a(s)ds.
\end{equation*}%
The change of variable $z=y(s)$ yields $B_{x_{0}}(y(t))= A_{t_0}(t)$.

Now we deal with Statement a). If $B_{x_{0}}(\infty )\geq A_{t_0}(\infty )$,
then $B_{x_{0}}(\infty )>A_{t_0}(t)$, for all $t>t_0$. Therefore $%
y(t)=B_{x_{0}}^{-1}(A_{t_0}(t)),$ $t>t_0$ is well-defined.

Finally we consider Statement b). In this case we have $%
B_{x_{0}}^{-1}(A_{t_0}(t))$ is only defined for $t<A_{t_0}^{-1}(B_{x_{0}}(%
\infty ))<\infty .\hfill $
\end{proof}

\bigskip

Also we are going to need the following elementary comparison result.

\begin{lemma}
\label{DesLemma}Let $x_{0}>r$ and $T>t_0$. Assume that \textbf{H1} and 
\textbf{H2} are satisfied, and that $u,v:[t_0,T] \rightarrow{\mathbb{R}}$
are two continuous functions.

\begin{itemize}
\item[a)] Suppose that $u$ and $v$ are such that 
\begin{eqnarray*}
v(t) &>&x_{0}+\int_{t_{0}}^{t}a(s)b(v(s))ds,\quad t\in \lbrack t_{0},T], \\
u(t) &=&x_{0}+\int_{t_{0}}^{t}a(s)b(u(s))ds,\quad t\in \lbrack t_{0},T].
\end{eqnarray*}%
Then $v(t)\geq u(t)$, for all $t\in \lbrack t_{0},T]$.

\item[b)] If 
\begin{eqnarray*}
r\ <\ v(t) &<&x_{0}+\int_{t_{0}}^{t}a(s)b(v(s))ds,\quad t\in \lbrack
t_{0},T], \\
u(t) &=&x_{0}+\int_{t_{0}}^{t}a(s)b(u(s))ds,\quad t\in \lbrack t_{0},T].
\end{eqnarray*}%
Then $v(t)\leq u(t)$, for all $t\in \lbrack t_{0},T]$.
\end{itemize}
\end{lemma}

\begin{proof}
We first deal with Statement a). Let $N=\{t\geq t_0:b(u(s))\leq b(v(s)),\
s\in \lbrack t_0,t]\}.$ Since $t_0\in N$, then the continuity of of $v$ and $%
u$, together with the fact that $b$ is non-decreasing on $(r,\infty)$, leads
us to show that $\tilde{T}=\sup N>t_0.$ If $\tilde{T}<T $ then 
\begin{equation*}
v(\tilde{T})-u(\tilde{T})>\int_{t_0}^{\tilde{T}}a(s)[b(v(s))-b(u(s))]ds \ge0,
\end{equation*}%
which is impossible due to the definition of ${\tilde T}$.

Finally, we proceed similarly to prove that b) is also true and to finish
the proof.$\hfill $
\end{proof}

\begin{theorem}
\label{TheOsnoiGral}Let $\xi \in{\mathbb{R}}$. Assume \textbf{H1}-\textbf{H3}%
. Then the explosion time $T_{\xi }^{X}$ of the solution $X^{\xi }$ of (\ref%
{eqdnhomo}) is finite if and only if 
\begin{equation}
\int_{r}^{\infty }\frac{ds}{b(s)}<\infty .  \label{conexisosgoog}
\end{equation}
\end{theorem}

\begin{proof}
Suppose that $T_{\xi }^{X}<\infty $. Since $g$ is continuous, then 
\begin{equation*}
\int_{0}^{t}a(s)b(X_{s}^{\xi })ds\left\{ 
\begin{array}{cc}
<\infty , & t<T_{\xi }^{X}, \\ 
=\infty , & t=T_{\xi }^{X}.%
\end{array}%
\right.
\end{equation*}%
Hence, there is $t_{0}\in (0,T_{\xi }^{X})$ such that 
\begin{equation*}
\xi +\int_{0}^{t_{0}}a(s)b(X_{s}^{\xi })ds+\inf_{s\in \lbrack 0,T_{\xi
}^{X}]}g(s)>r,
\end{equation*}%
and consequently $X_{t}>r$ for $t\in \lbrack t_{0},T_{\xi }^{X}]$.

Now set 
\begin{equation*}
M=\sup \{|g(t)|:0\leq t\leq T_{\xi
}^{X}\}+\xi+\int_0^{t_0}a(s)b(X_s^{\xi})ds .
\end{equation*}%
This yields%
\begin{equation*}
X_{t}^{\xi }< M+1+\int_{t_0}^{t}a(s)b(X_{s}^{\xi })ds,\ \ t\in \lbrack
t_0,T_{\xi }^{X}].
\end{equation*}%
On the other hand, we consider the integral equation%
\begin{equation*}
u(t)=(M+1)+\int_{t_0}^{t}a(s)b(u(s))ds,\ \ t\geq t_0.
\end{equation*}%
Because $M>r$, Lemmas \ref{lem:n7} and \ref{DesLemma} give $T_{M+1}^{u}=
A_{t_0}^{-1}(B_{M+1}(\infty ))\leq T_{\xi }^{X}<\infty $. Whence%
\begin{equation*}
\int_{M+1}^{\infty }\frac{ds}{b(s)}<\infty .
\end{equation*}%
The continuity and positivity of $b$ in $[r,\infty )$ implies (\ref%
{conexisosgoog}).

Reciprocally, suppose that $X^{\xi }$ does not explodes in finite time. From
Hypotheses \textbf{H1} and \textbf{H3}, we can find a sequence $\{t_{n}: n\in%
{\mathbb{N}}\}$ such that $t_{n}\uparrow \infty $ and 
\begin{equation*}
r+1<\xi +\inf_{0\leq h\leq {\tilde \eta} }g(t_{n}+h)\uparrow \infty \text{,\
\ as }n\rightarrow \infty.
\end{equation*}
Observe that 
\begin{equation*}
X_{t+t_{n}}^{\xi }>\xi +\inf_{0\leq h\leq {\tilde\eta} }g(t_{n}+h)-1+%
\int_{0}^{t}a(s+t_{n})b(X_{s+t_{n}}^{\xi })ds,\ \ t\in \lbrack 0,{\tilde\eta}
].
\end{equation*}%
Now consider the integral equation%
\begin{equation*}
u(t)=\xi +\inf_{0\leq h\leq {\tilde\eta} }g(t_{n}+h)-1+%
\int_{0}^{t}a(s+t_{n})b(u(s))ds,\ \ t\in[0,{\tilde\eta}].
\end{equation*}%
Therefore Lemmas \ref{lem:n7} and \ref{DesLemma} yield%
\begin{equation*}
\int_{\xi +\inf_{0\leq h\leq {\tilde\eta} }g(t_{n}+h)-1}^{\infty }\frac{ds}{%
b(s)}>\int_{t_{n}}^{t_n+{\tilde \eta} }a(s)ds.
\end{equation*}%
Whence \textbf{H1} implies $\int_{r}^{\infty }\frac{ds}{b(s)}=\infty .\hfill 
$
\end{proof}

We finish this section with the following result for bounded noise.

\begin{proposition}
\label{BNoise} Assume that Hypotheses \textbf{H1} and \textbf{H2} are true.
Also assume that $g$ in equation (\ref{eqdnhomo}) is a bounded function and
that $\xi +\inf_{s\geq 0}g(s)>r$. Then, we have the following statements:

\begin{itemize}
\item[a)] $\int_r^{\infty} (1/b(s))ds=\infty$ implies that the solution of
equation (\ref{eqdnhomo}) does not explode in finite time.

\item[b)] $\int_r^{\infty} (1/b(s))ds<\infty$ yields that the solution of
equation (\ref{eqdnhomo}) blows up in finite time and 
\begin{equation*}
T_{\xi }^{X}\in (A_0^{-1}(B_{\xi +\sup_{s\geq 0}g(s)}(\infty
)),A^{-1}_0(B_{\xi +\inf_{s\geq 0}g(s)}(\infty ))).
\end{equation*}
\end{itemize}
\end{proposition}

\begin{proof}
Let $\varepsilon >0$ be such that $\xi +\inf_{s\geq 0}R(s)>r+\varepsilon $.
Set 
\begin{equation*}
Z_{t}^{\xi }=\xi +\sup_{s\geq 0}g(s)+\varepsilon
+\int_{0}^{t}a(s)b(Z_{s}^{\xi })ds
\end{equation*}%
and 
\begin{equation*}
Y_{t}^{\xi }=\xi +\inf_{s\geq 0}g(s)-\varepsilon
+\int_{0}^{t}a(s)b(Y_{s}^{\xi })ds.
\end{equation*}%
By Lemma \ref{DesLemma} we have, 
\begin{equation*}
Y_{t}^{\xi }<X_{t}^{\xi }<Z_{t}^{\xi },\quad t<T_{\xi +\sup_{s\geq
0}g(s)+\varepsilon }^{Z}.
\end{equation*}%
Letting $\varepsilon \downarrow 0$ the proof is an immediate consequence of Lemma \ref{lem:n7}, and
Hypotheses \textbf{H1} and \textbf{H2}. \hfill 
\end{proof}

\section{Stochastic differential equation with additive Wiener integral noise}

\label{sec:4}

In this section we study equation (\ref{eqdnhomo}) when the noise $g$ is a
Wiener integral. More precisely, here we study the stochastic differential
equation 
\begin{equation}
X_{t}^{\xi }=\xi +\int_{0}^{t}a(s)b(X_{s}^{\xi })ds+I_{t},
\label{sdeintnoise}
\end{equation}%
where $I_{t}=\int_{0}^{t}f(s)dW_{s}$ and $f:[0,\infty )\rightarrow \mathbb{R}
$ is a square-integrable function on $[0,M]$, for any $M>0$.

In the remaining of this section we utilize the following assumption:

\begin{description}
\item[\textbf{H4}:] $\int_{0}^{\infty }f^{2}(s)ds=\infty $ and 
\begin{equation}
\sum_{n=M}^{\infty}\frac{1}{\Upsilon^p (n)}\left( \int_{n}^{n+2
}f^{2}(s)ds\right) ^{p/2}<\infty,  \label{condconvseri}
\end{equation}%
for some $M,p>0$, where 
\begin{equation*}
\Upsilon (t)=\sqrt{2\left( \int_{0}^{t}f^{2}(s)ds\right) \log \log \left(
e^e\vee \int_{0}^{t}f^{2}(s)ds\right) }.
\end{equation*}
\end{description}

\begin{remark}
\label{rem:nic} Observe that (\ref{condconvseri}) holds if, for example, 
\begin{equation*}
t\mapsto \left( \int_{0}^{t+2}f^{2}(s)ds\right) \left(
\int_{0}^{t}f^{2}(s)ds\right) ^{-1}-1.
\end{equation*}%
is a decreasing function in $L^{p}([M,\infty ))$ for some $M,p>0$.
\end{remark}

On the other hand, as a consequence of iterated logarithm theorem for
locally square integrable martingales, we can now state the following:

\begin{lemma}
Under the fact that $\int_0^{\infty}f^2(s)ds=\infty$, we have 
\begin{equation}
\limsup_{t\rightarrow \infty }\frac{I_{t}}{\Upsilon (t)}=1 \quad 
\hbox{\rm
with probability one}.  \label{lit}
\end{equation}
\end{lemma}

\begin{proof}
The result is Theorem 1.1 in Qing Gao \cite{G}.$\hfill $
\end{proof}

The following theorem is the main result of this section. 

\begin{theorem}
\label{TRI}Assume that \textbf{H1}, \textbf{H2} and \textbf{H4} are true.
Then the stochastic differential equation (\ref{sdeintnoise}) blows up in
finite time with probability 1 if and only if $\int_{r}^{\infty }\frac{ds}{%
b(s)}<\infty .$
\end{theorem}

\begin{proof}
We first observe that, by Theorem \ref{TheOsnoiGral}, we only need to show
that the paths of $I$ satisfy Hypothesis \textbf{H3} almost surely.

Burkholder-Davis-Gundy inequality (see, for instance, Theorem 3.5.1 in \cite%
{Du}) yields%
\begin{equation*}
E\left[ \left( \sup_{s,t\in \lbrack {n},{n}+2 ]}|I_{t}-I_{s}|\right) ^{p}%
\right] \leq c_{p}\left( \int_{{n}}^{{n}+2}f^{2}(s)ds\right) ^{p/2},
\end{equation*}%
where $c_{p}$ is a constant depending only on $p$. Then, by (\ref%
{condconvseri}),

\begin{equation*}
E\left[ \sum_{n=M}^{\infty }\left( \sup_{s,t\in \lbrack {n},n+2]}\frac{%
|I_{t}-I_{s}|}{\Upsilon ({n})}\right) ^{p}\right] \leq
c_{p}\sum_{n=M}^{\infty }\frac{1}{\Upsilon ^{p}({n})}\left( \int_{{n}}^{{n}%
+2}f^{2}(s)ds\right) ^{p/2}<\infty .
\end{equation*}%
Therefore, it is enough to prove that $I(\omega )$ satisfies \textbf{H3} for 
$\omega \in \Omega $ for which there exists $n_{0}\in \mathbb{N}$ such that%
\begin{equation*}
\sup_{s,t\in \lbrack {n},{n}+2]}\frac{|I_{t}(\omega )-I_{s}(\omega )|}{%
\Upsilon ({n})}\leq \frac{1}{4},\quad \hbox{\rm for}\ n\geq n_{0}
\end{equation*}%
and (\ref{lit}) is satisfied. Hence, we can find a sequence $\{t_{n}:n\in {%
\mathbb{N}}\}$ such that $t_{n}>n$ and%
\begin{equation*}
\frac{I_{t_{n}}(\omega )}{\Upsilon (t_{n})}\geq \frac{1}{2}\quad 
\hbox{\rm
for all}\ n\in {\mathbb{N}}.
\end{equation*}%
Finally, using the properties established in this proof, we are able to
write, for $n\geq n_{0}$, 
\begin{eqnarray*}
\inf_{s\in \lbrack t_{n},t_{n}+1]}I_{s}(\omega ) &=&I_{t_{n}}(\omega
)+\inf_{s\in \lbrack t_{n},t_{n}+1]}\left( I_{s}(\omega )-I_{t_{n}}(\omega
)\right) \\
&\geq &I_{t_{n}}(\omega )+\inf_{s\in \lbrack t_{n},t_{n}+1]}(-|I_{s}(\omega
)-I_{t_{n}}(\omega )|) \\
&\geq &I_{t_{n}}(\omega )-\left( \sup_{s,t\in \lbrack \lbrack
t_{n}],[t_{n}]+2]}\frac{|I_{s}(\omega )-I_{t}(\omega )|}{\Upsilon ([t_{n}])}%
\right) \Upsilon ([t_{n}]) \\
&\geq &\frac{1}{2}\Upsilon (t_{n})-\frac{1}{4}\Upsilon ([t_{n}])\geq \frac{1%
}{4}\Upsilon (t_{n})\rightarrow \infty ,
\end{eqnarray*}%
as $n\rightarrow \infty $, where $[t]$ is the integer part of $t$ and, in
the last inequality, we have used that $\Upsilon $ is a non-decreasing
function.$\hfill $
\end{proof}

\bigskip

Now, in order to state a consequence of Theorem \ref{TRI}, we consider the
equation 
\begin{equation}  \label{eq:nagre}
Y_t=\xi+\int_0^t{\tilde b}(s,Y_s)ds+I_t,\quad t\ge 0.
\end{equation}
Here, for each $T>0$, the function ${\tilde b}:[0,\infty))\times{\mathbb{R}}%
\rightarrow [0,\infty)$ is locally Lipschitz (uniformly on $s\in[0,T]$), $%
b(\cdot,x)$ is continuous, for $x\in{\mathbb{R}}$, and $I$ satisfy
Hypothesis \textbf{H4} with $f$ continuous. Remember that, in this case,
equation (\ref{eq:nagre}) has a unique solution that may explode in finite
time.

\begin{corollary}
Let $a$ and $b$ satisfy Conditions \textbf{H1} and \textbf{H2},
respectively. Assume that $\xi\in{\mathbb{R}}$, $b$ is locally Lipschitz, $%
\int_r^\infty (1/b(x))dx<\infty$ (resp. $\int_r^\infty (1/b(x))dx=\infty$)
and $a(s)b(x)\le {\tilde b}(s,x)$ (resp. ${\tilde b}(s,x)\le a(s)b(x)$), $%
(s,x)\in[0,\infty)\times{\mathbb{R}}$. Then, the solution to equation (\ref%
{eq:nagre}) explodes (resp. does not explode) in finite time.
\end{corollary}

\begin{proof}
We only consider the case that $\int_r^\infty (1/b(x))dx=\infty$ and ${%
\tilde b}(s,x)\le a(s)b(x)$, since the proof is similar for the other one.

Let $X^{\xi }$ and $Y$ be the solutions of equations (\ref{sdeintnoise}) and
(\ref{eq:nagre}), respectively. Then, from Milian \cite{Mi} (Theorem 2), we
get 
\begin{equation*}
Y_{t}\leq X_{t}^{\xi },\quad t\geq 0.
\end{equation*}%
Thus, by Theorem \ref{TRI}, the solution $Y$ of equation (\ref{eq:nagre})
cannot explode in finite time because it cannot go to $-\infty $ in finite
time since ${\tilde{b}}$ is ${\mathbb{R}}_{+}$-valued and $I$ has continuous
paths and, consequently, bounded paths on compact intervals of $[0,\infty )$%
. Therefore the proof is complete.$\hfill $
\end{proof}

\begin{example}
Take 
\begin{eqnarray*}
a(x) &=&x^{\alpha },\ \ x\in (0,\infty ), \\
b(x) &=&8x^{2}-36x+48,\ \ x\in \mathbb{R}, \\
f(x) &=&x^{\beta },\ \ x\in (0,\infty ),\ \ \beta >-\frac{1}{2}.
\end{eqnarray*}%
Hence%
\begin{equation*}
\lim_{t\rightarrow \infty }\int_{t}^{t+1}x^{\alpha }dx=\left\{ 
\begin{array}{cc}
+\infty , & \alpha >0, \\ 
1, & \alpha =0, \\ 
0, & \alpha <0,%
\end{array}%
\right. 
\end{equation*}%
and 
\begin{equation*}
\frac{(t+2)^{2\beta +1}}{t^{2\beta +1}}-1=\left( 1+\frac{2}{t}\right)
^{2\beta +1}-1\leq C\frac{1}{t}.
\end{equation*}%
The last function belongs to $L^{p}([1,\infty ])$, for any $p>1$. Thus $f$
satisfied (\ref{condconvseri}) due to Remark \ref{rem:nic}.

On the other hand, it is clear that $\int_{\xi }^{\infty }\frac{dx}{%
8x^{2}-36x+48}<\infty $, $\xi >0$. Then%
\begin{equation*}
X_{t}^{\xi }=\xi +\int_{0}^{t}s^{\alpha }(8(X_{s}^{\xi })^{2}-36(X_{s}^{\xi
})+48)ds+\int_{0}^{t}s^{\beta }dW_{s},
\end{equation*}%
explodes in finite time when $\alpha \geq 0$. Notice that $b$ is not
necessarily increasing as in \cite{L-V} or \cite{M-J-J}. Moreover, we can
improve Theorem \ref{TRI} in some particular cases, see \cite{V}.
\end{example}

\begin{example}
\label{ExEXSC} The function $Y_{t}\equiv 1$ is solution to%
\begin{equation*}
Y_{t}=1+\int_{0}^{t}(Y_{s})^{2}ds-t,\ \ t\geq 0.
\end{equation*}%
Although $\int_{1}^{\infty }(1/s^{2})ds<\infty $, $Y$ does not blow-up in
finite time because $g(t)=-t,\ t\geq 0,$ does not satisfies Hypothesis 
\textbf{H3}.
\end{example}

Also notice that $f(t)=\exp (\exp (t))$, $t\geq 0,$ does not satisfies (\ref%
{condconvseri}). We intuitively understand that in this case the noise is to
strong and we have also blow up in finite time, for any initial condition.
We have a contrary effect as in Example \ref{ExEXSC}.

\begin{proposition}
Let $f$ and $I$ be defined in equation (\ref{sdeintnoise}). Suppose \textbf{%
H1}, \textbf{H2} and $\int_{0}^{\infty }f^{2}(s)ds<\infty $ are satisfied.
Then $I$ is bounded with probability one and, under the assumption $\xi
+\inf_{s\geq 0}I_{s}> r$, the stochastic differential equation (\ref%
{sdeintnoise}) blows up in finite time if and only if $B_{r}(\infty )
<\infty $.
\end{proposition}

\noindent\textbf{Remark} Observe that $\xi +\inf_{s\geq 0}I_{s}$ depends on $%
\omega$.

\begin{proof}
The result follows from \cite{Du} (Lemma 3.4.7 and Theorem 3.4.9), and
Proposition \ref{BNoise}.$\hfill $
\end{proof}

\section{An approach to obtain the distribution of the explosion time of a
stochastic differential equation}

\label{sec:5}

Now we study some stochastic differential equations of the form 
\begin{equation}
X_{t}^{\xi }=\xi +\int_{0}^{t}b(s,X_{s}^{\xi })ds+\int_{0}^{t}\sigma
(s,X_{s}^{\xi })dW_{s},\quad t\ge 0.  \label{edegral}
\end{equation}%
Namely, we propose a method to figure out the distribution of the explosion
time $\tau _{\xi }$ of $X^{\xi}$. Intuitively, $\tau_{\xi}$ is a stopping
time such that (\ref{edegral}) has a solution up to this stopping time and $%
\limsup_{t\uparrow \tau_{\xi}}|X_t|=\infty.$

\subsection{Autonomous case}

\label{subsec:5.1} This section is devoted to deal with the stochastic
differential equation 
\begin{equation*}
X_{t}^{\xi }=\xi +\int_{0}^{t}b(X_{s}^{\xi })ds+\int_{0}^{t}\sigma
(X_{s}^{\xi })dW_{s},\quad t\geq 0,
\end{equation*}%
with $b,\sigma \in C^{1}({\mathbb{R}})$. In this case, McKean \cite{MK} has
shown that $X_{(\tau _{\xi })-}^{\xi }\in \{-\infty ,\infty \}$ on $[\tau
_{\xi }<\infty ]$. So, henceforth, we can utilize the convention 
\begin{equation*}
\tau _{\xi }^{+}=\inf \{t>0:X_{t}^{\xi }=\infty \}\quad \hbox{\rm and}\quad
\tau _{\xi }^{-}=\inf \{t>0:X_{t}^{\xi }=-\infty \}.
\end{equation*}

\begin{theorem}
\label{thm:distet} Consider a bounded function $u:[0,\infty )\times \mathbb{R%
}\rightarrow \mathbb{R}$ that satisfies the following boundary value problem:%
\begin{eqnarray}
\frac{\partial u}{\partial t}(t,x) &=&\frac{1}{2}\sigma ^{2}(x)\frac{%
\partial ^{2}u}{\partial x^{2}}(t,x)+b(x)\frac{\partial u}{\partial x}(t,x),
\quad t> 0\ \hbox{\rm and}\ x\in{\mathbb{R}},  \label{eccalor} \\
u(0,x) &=&0,\quad \hbox{\rm for all}\ x\in \mathbb{R}.  \label{conespac}
\end{eqnarray}

\begin{itemize}
\item[a)] Assume that $u(t,\infty)=u(t,-\infty)=1$. Then $P(\tau_{\xi}\le
t)=u(t,\xi).$

\item[b)] $u(t,\infty)=1$ and $u(t,-\infty)=0$ implies that $%
P(\tau_{\xi}^+\le t)=u(t,\xi).$

\item[c)] If $u(t,\infty)=0$ and $u(t,-\infty)=1$, we have $%
P(\tau_{\xi}^-\le t)=u(t,\xi).$
\end{itemize}
\end{theorem}

\noindent\textbf{Remarks}\textit{
\begin{itemize}
\item[1)] Maximum principle  provides
with conditions on $b$ and $\sigma$ that guarantee  the 
solution of equation (\ref{eccalor}) is bounded (see Friedman \cite{F}).
\item[2)]
It is quite interesting to observe that (\ref{eccalor}) is related to
transition density of process $X^{\xi}$, or related to the fundamental
solution of the associated Cauchy problem (see \cite{F}). On the other hand,
(\ref{conespac}) and the conditions in Statement a)-c)
are intuitively clear. In fact, (\ref%
{conespac}) establishes that if we begin at a real point ($\xi\in \mathbb{R}$),
then we need some time to get blow-up. And other conditions  mean
 that if we
begin at cementery state ($\pm \infty $), then the time to blow-up is less
than any time.
\item[3)] Observe that $P(\tau^{\xi}\le t)=P(\tau^{\xi}_+\le 
t)+P(\tau^{\xi}_-\le t)$ and that, for example in Statement a), we
have $P(\tau^{\xi}<\infty)=u(\infty,\xi).$
\item[4)] If $X^{\xi}$ does not explodes in finite time, then equation
(\ref{eccalor})-(\ref{conespac}) has not a bounded solution
satisfying  conditions established in either Statement a), b), or c).
\end{itemize}}

\begin{proof}
Using It\^{o}'s formula on $0\leq s<t$ and that $u$ is solution to (\ref%
{eccalor})\ we obtain%
\begin{equation*}
u(t-(s\wedge \tau _{\xi }^{m}),X_{s\wedge \tau _{\xi }^{m}}^{\xi })=u(t,\xi
)+\int_{0}^{s\wedge \tau _{\xi }^{m}}\frac{\partial u}{\partial x}%
(t-r,X_{r}^{\xi })dW_{r},
\end{equation*}%
where $\tau _{\xi }^{m}=\inf \{t>0:|X_{t}^{\xi }|>m\}$. Since $u$ is
bounded, then the above stochastic integral is a martingale. Therefore 
\begin{equation*}
u(t,\xi )=E\left[ u(t-(s\wedge \tau _{\xi }^{m}),X_{s\wedge \tau _{\xi
}^{m}}^{\xi })\right] .
\end{equation*}%
Letting $s\uparrow t$, then continuity of $X^{\xi }$ and the boundedness of $%
u$, together with the dominated convergence theorem, allow us to write 
\begin{eqnarray*}
u(t,\xi ) &=&E\left[ u(t-(t\wedge \tau _{\xi }^{m}),X_{t\wedge \tau _{\xi
}^{m}}^{\xi })\right] \\
&=&E\left[ u(t-(t\wedge \tau _{\xi }^{m}),X_{t\wedge \tau _{\xi }^{m}}^{\xi
}),\tau _{\xi }\leq t\right] \\
&&+E\left[ u(t-(t\wedge \tau _{m}^{\xi }),X_{t\wedge \tau _{m}^{\xi }}^{\xi
}),\tau _{\xi }>t\right] .
\end{eqnarray*}%
Taking $m\rightarrow \infty ,$%
\begin{equation}
u(t,\xi )=E\left[ u(t-\tau _{\xi },X_{\tau _{\xi }}^{\xi }),\tau _{\xi }\leq
t\right] +E\left[ u(0,X_{t}^{\xi }),\tau _{\xi }>t\right]  \label{distrexp}
\end{equation}%
Now we consider Statement a),
\begin{equation*}
u(t,\xi )=E\left[ u(t-\tau _{\xi },X_{\tau _{\xi }}^{\xi }),\tau _{\xi }\leq
t\right] =P(\tau _{\xi }\leq t).
\end{equation*}%
Statement b) is proven as follows. From equality (\ref{distrexp}) we get 
\begin{eqnarray*}
u(t,\xi ) &=&E\left[ u(t-\tau _{\xi },X_{\tau _{\xi }}^{\xi }),\tau _{\xi
}\leq t,\tau _{\xi }^{+}<\tau _{\xi }^{-}\right] \\
&&+E\left[ u(t-\tau _{\xi },X_{\tau _{\xi }}^{\xi }),\tau _{\xi }\leq t,\tau
_{\xi }^{-}<\tau _{\xi }^{+}\right] \\
&=&P(\tau _{\xi }^{+}\leq t).
\end{eqnarray*}

Finally, Statement c) is proven similarly. So the proof is complete. $\hfill 
$
\end{proof}

\begin{examples}
\begin{itemize}
\item[a)] In Example \ref{ExamOSG1}, with $f\equiv 1$, we have%
\begin{equation*}
\bar{u}(t,\xi )=\left\{ 
\begin{array}{ll}
\Phi \left( \frac{|\xi |^{1-\alpha }}{(\alpha -1)\sqrt{t}}\right) , & \xi >0,
\\ 
0, & \xi \leq 0,%
\end{array}%
\right.
\end{equation*}%
and 
\begin{equation*}
\text{\b{u}}(t,\xi )=\left\{ 
\begin{array}{ll}
0, & \xi >0, \\ 
\Phi \left( \frac{|\xi |^{1-\alpha }}{(\alpha -1)\sqrt{t}}\right) , & \xi
\leq 0,%
\end{array}%
\right.
\end{equation*}%
satisfy Statements b) and c) of Theorem \ref{thm:distet}, respectively. In
particular, if $\xi >0$, then $\bar{u}(\infty ,\xi )=1$, therefore we have a
positive blow-up. This phenomenon is explained by Milan \cite{Mi} (Theorem
1) due to the involved solution being a nonnegative process when $\xi >0$.

\item[b)] For $\beta >0$, the partial differential equation%
\begin{eqnarray*}
\frac{\partial u}{\partial t}(t,x) &=&\frac{a^{2}}{2}e^{2\beta x}\frac{%
\partial ^{2}u}{\partial x^{2}}(t,x)+\beta a^{2}e^{2\beta x}\frac{\partial u%
}{\partial x}(t,x), \\
u(0,x) &=&0,\ \ \forall x\in \mathbb{R},
\end{eqnarray*}%
has solution,%
\begin{equation*}
u(t,x)=\exp \left( -\frac{e^{-2\beta x}}{2(a\beta )^{2}t}\right) .
\end{equation*}%
Since $u(t,\infty )=1$ and $u(t,-\infty )=0$, then the distribution of
explosion time to the stochastic differential equation%
\begin{equation*}
X_{t}^{x}=x+\int_{0}^{t}\beta a^{2}e^{2\beta
X_{s}^{x}}ds+\int_{0}^{t}ae^{\beta X_{s}^{x}}dW_{s},
\end{equation*}%
is given by 
\begin{equation*}
P(\tau _{\xi }\leq t)=\exp \left( -\frac{e^{-2\beta x}}{2(a\beta )^{2}t}%
\right)
\end{equation*}%
because of $P(\tau _{\xi }^{+}<\infty )=1$.
\end{itemize}
\end{examples}

\bigskip

\noindent\textbf{Remark}\textit{\ It is not difficult to see that Examples %
\ref{ExamOSG1} and \ref{ExamOSG2} are solution of the corresponding partial
differential equations (PDEs), then we conjecture that the distribution of
the explosion time is the solution of such a PDE. If this is true, then we
have the following criterion of explosion: There is explosion in finite time
if and only if the corresponding PDE has a bounded solution. Moreover, this
criterion could be applied in more dimensions and for non autonomous
processes (see \cite{F}).}

\subsection{Laplace transform of the distribution of the explosion time}

\label{sec:5.2} Finally, in this subsection we indicate how we could calculate
the Laplace transformation of the distribution of the explosion time $%
\tau_{\xi}$ of the solution to equation (\ref{eq:nagre}). It means, we
assume that the equation 
\begin{equation}  \label{eq:lamisma}
X^{\xi}_t=\xi+\int_0^t{\ b}(s,X^{\xi}_s)ds+I_t,\quad t\ge 0,
\end{equation}
has a unique solution that may blow-up in finite time, where $b$ takes
values in ${\mathbb{R}}_+$ and $I_t=\int_0^tf(s)dW_s$. Note that if $%
\omega\in\Omega$ is such that $\tau_{\xi}(\omega)<\infty$, then $%
X^{\xi}_t>\xi+\inf_{0\le s\le \tau_{\xi}(\omega)}I_t(\omega),\ t\le
\tau_{\xi}(\omega)$, and consequently $\int_0^{\tau_{\xi}(\omega)}{\ b}%
(s,X^{\xi}_s)ds=\infty$. Thus, $X^{\xi}_{\tau_{\xi}-}=\infty$, on $%
[\tau_{\xi}<\infty]$, and $\tau_{\xi}=\tau_{\xi}^+$.

We begin with an auxiliary result.

\begin{lemma}
\label{lem:tl} let $\lambda >0$ and $\tau _{\xi }$ the explosion time of the
solution of equation (\ref{eq:lamisma}). Then 
\begin{equation*}
E\left( e^{-\lambda \tau _{\xi }}\right) =\lambda \int_{0}^{\infty }P(\tau
_{\xi }\leq u)e^{-\lambda u}du.
\end{equation*}
\end{lemma}

\begin{proof}
Let us denote the distribution of $\tau _{\xi }$ by $F_{\tau _{\xi }}$.
Fubini theorem leads to justify 
\begin{eqnarray*}
E\left( e^{-\lambda \tau _{\xi }}\right) &=&\int_{(0,\infty )}e^{-\lambda
s}F_{\tau _{\xi }}(ds)=\lambda \int_{(0,\infty )}\left( \int_{s}^{\infty
}e^{-\lambda u}du\right) F_{\tau _{\xi }}(ds) \\
&=&\lambda \int_{0}^{\infty }\left( \int_{(0,u]}F_{\tau _{\xi }}(ds)\right)
e^{-\lambda u}du=\lambda \int_{0}^{\infty }F_{\tau _{\xi }}(u)e^{-\lambda
u}du.
\end{eqnarray*}%
Consequently, the proof is complete.$\hfill $
\end{proof}

Now we can state the main result of this subsection.

\begin{theorem}
\label{thm:22} Consider $\lambda>0$, the explosion time $\tau_{\xi}$ of the
solution of (\ref{eq:lamisma}) and a bounded function $u:[0,\infty )\times 
\mathbb{R}\rightarrow \mathbb{R}$ that is a solution of the partial
differential equation 
\begin{eqnarray*}  \label{eq:tlas}
-\frac{\partial u}{\partial t}(t,x) &=&\frac{1}{2}f^{2}(t)\frac{\partial
^{2}u}{\partial x^{2}}(t,x)+b(t,x)\frac{\partial u}{\partial x}(t,x)-\lambda
u(t,x), \ t>0\ \hbox{\rm and}\ x\in{\mathbb{R}}, \\
u(t,\infty)&=&1.
\end{eqnarray*}%
Then, $\lambda\int_0^{\infty} P(\tau_{\xi}\le u)e^{-\lambda u}du =u(0,\xi).$
\end{theorem}

\begin{proof}
As in the proof of Theorem \ref{thm:distet} It\^{o}'s formula gives 
\begin{equation*}
u(0,\xi )=E\left( u(t\wedge \tau _{\xi }^{m},X_{t\wedge \tau _{\xi
}^{m}}^{\xi })\exp (-\lambda (t\wedge \tau _{\xi }^{m}))\right) .
\end{equation*}%
We can now easily complete the proof of this result by combining Lemma \ref%
{lem:tl}, the arguments used in the last part of the proof of Theorem \ref%
{thm:distet} and the fact that $X_{\tau _{\xi }^{m}}^{\xi }\rightarrow
\infty $ as $m\rightarrow \infty $ on $[\tau _{\xi }<\infty ]$. Indeed we
first take $m\uparrow \infty $, and then $t\uparrow \infty $. $\hfill $
\end{proof}

\bigskip 

\noindent \textbf{Remark}\textit{\ In some cases we have the converse of
Theorem \ref{thm:22}. For example, consider the stochastic differential
equation
\begin{equation*}
X_{t}^{\xi }=\xi +\int_{0}^{t}(g(X_{s}^{\xi })+a-b)ds+cW_{t},\ \ t\geq 0,
\end{equation*}%
where $a,b,c\in \mathbb{R}$, $a>b$ and $g:{\mathbb{R}} \rightarrow \lbrack 0,\infty )$. Then the associated ordinary
differential equation is
\begin{eqnarray}
\frac{c^{2}}{2}w^{\prime \prime }(x)+(g(x)+a-b)w^{\prime }(x)-\lambda w(x)
&=&0,\ \ t>0, \label{exexpfeller} \\
w(\infty ) &=&1.  \notag
\end{eqnarray}
Therefore, if $X^{\xi }$ explodes in finite then (\ref{exexpfeller}) has a bounded solution, in fact it is the Laplace transform
of the explosion time (see \cite{Wfe}). Then the solution of 
\begin{eqnarray*}
-\frac{\partial u}{\partial t}(t,x) &=&\frac{c^{2}}{2}\frac{\partial ^{2}u}{%
\partial x^{2}}(t,x)+(g(x-bt)+a)\frac{\partial u}{\partial x}(t,x)-\lambda
u(t,x),\ \ t>0, \ x\in {\mathbb{R}}, \\
u(t,\infty ) &=&1.
\end{eqnarray*}%
is given by $u(t,x)=w(x-bt)$.}

\bigskip

\noindent \textbf{Acknowledgements:} Professor Villa-Morales was partially
supported by grant 118294 of CONACyT and grant PIM13-3N of UAA. The authors
Thanks Universidad Aut\'{o}noma de Aguascalientes y Cinvestav-IPN for their
hospitality and economical support. Other authors were also partially
supported by a CONACyT grant.

\end{document}